\documentclass[12pt]{amsart}

\usepackage{amsmath, amssymb}
\usepackage{tikz} 
\usepackage{accents}
\usepackage{euscript}

\newcommand{\fX}{\mathfrak X}
\newcommand{\doo}{\diamondsuit_{\aleph_1}}
\DeclareMathOperator{\SEQ}{\{0,1\}^{<\bbN}}

\DeclareMathOperator{\xSEQ}{\{0,1\}^{\bbN}}
\newcommand{\NULL}{\langle\rangle}

\newcommand{\cS}{{\mathcal S}}
\newcommand{\cO}{{\mathcal O}}
\newcommand{\cstar}{$C^*$}

\newcommand{\N}{{\mathbb N}}

\newcommand{\bbR}{{\mathbb R}}

\newcommand{\bbN}{{\mathbb N}}
\newcommand{\bbC}{\mathbb C}
\newcommand{\bbQ}{\mathbb Q}

\newcommand{\bbX}{{\mathbb  X}}
\newcommand{\cX}{{\mathcal X}}
\newcommand{\cW}{{\mathcal W}}

\newcommand{\cY}{{\mathcal Y}}
\newcommand{\cZ}{{\mathcal W}}
\newcommand{\cA}{{\mathbb A}}

\newcommand{\rs}[1]{\restriction_{#1}}

\DeclareMathOperator{\id}{id}

\newcommand{\cC}{\mathcal C}
\newcommand{\cF}{\mathcal F}

\newcommand{\e}{\varepsilon}

\newtheorem{thm}{Theorem}[section]
\newtheorem{theorem}[thm]{Theorem}

\newtheorem{remark}[thm]{Remark}

\newtheorem{lemma}[thm]{Lemma}

\theoremstyle{definition}

\newtheorem{definition}[thm]{Definition}

\newcounter{my_enumerate_counter}
\newcommand{\pushcounter}{\setcounter{my_enumerate_counter}{\value{enumi}}}
\newcommand{\popcounter}{\setcounter{enumi}{\value{my_enumerate_counter}}}

\def\rs{\restriction}

\newcommand{\cP}{{\mathcal P}}

\newcommand{\bbZ}{\mathbb Z}

\DeclareMathOperator{\Ad}{Ad}

\DeclareMathOperator{\Tr}{Tr}

\title{A simple AF algebra not isomorphic to its opposite}
\author{Ilijas Farah} 

\address{Department of Mathematics and Statistics\\
York University\\
4700 Keele Street\\
North York, Ontario\\ Canada, M3J 1P3}
\email{ifarah@yorku.ca}
\urladdr{http://www.math.yorku.ca/$\sim$ifarah}

\author{Ilan Hirshberg} 

\address{Department of Mathematics\\
 Ben Gurion University of the Negev\\
  P.O.B. 653, Be'er\\
Sheva 84105, Israel}
\email{ilan@math.bgu.ac.il}

\urladdr{http://www.math.bgu.ac.il/~ilan/}

\dedicatory{Dedicated to Menachem Magidor on the occasion of his 70th birthday.} 

\date{\today}
\thanks{This research was supported by the Israel Science Foundation, grant number 476/16 (IH) and NSERC (IF). Most of the work on this paper was done when IH was visiting IF at the Fields Institute in September 2016.}

\date{\today}
\begin{document}

\begin{abstract}
We show that it is consistent with ZFC that there is a simple nuclear non-separable $C^*$-algebra which is not isomorphic to its opposite algebra. We can furthermore guarantee that this example is an inductive limit of unital 
copies of the Cuntz algebra $\cO_2$, or of the CAR algebra.
\end{abstract}
\maketitle
\subsection*{Significance statement}
The Hilbert space $\ell^2$ is the (usually infinite-dimensional) modification of our standard three-dimensional  space. \cstar-algebras are suitably closed algebras of linear operators on $\ell^2$. The algebras of complex $n\times n$  matrices are the simplest examples of \cstar-algebras. The opposite of a \cstar-algebra is the algebra in which the direction of the multiplication is reversed. 
Although every matrix algebra is isomorphic to its opposite,  we construct an inductive limit of matrix algebras not isomorphic to its opposite. This is the first known example of a simple amenable \cstar-algebra not isomorphic to its opposite. Our examples can have exactly $n$ inequivalent irreducible representations for any $n$, showing that Glimm's dichotomy can fail for simple nonseparable \cstar-algebras.

\section{Introduction} 
%The purpose of this paper is to show that it is consistent with ZFC that there exists a simple nuclear  \cstar-algebra which is not isomorphic to its opposite. Our examples are non-separable, and can be arranged to furthermore be (transfinite) unital inductive limits of copies of the Cuntz algebra $\cO_2$ or of copies of the CAR algebra $M_{2^{\infty}}$ whose connecting maps induce the identity on $K_0$, in which case they have either trivial $K$-theory, or $K$-theory isomorphic to that of the CAR algebra.
%
The opposite algebra of a \cstar-algebra $A$ is the \cstar-algebra whose underlying Banach space structure and involution are the same as that of $A$, but  the product of $x$ and $y$ is defined as $yx$ rather than $xy$.
It is  denoted by $A^{\mathrm{op}}$.  
 In~\cite{connes} Connes constructed 
 examples of factors not isomorphic to their opposites. Phillips used Connes' results in \cite{phillips-PAMS} to construct simple separable examples, and Phillips--Viola in \cite{phillips-viola} improved this to construct a simple separable exact example. In the nuclear setting, one can construct non-simple examples (\cite{rosenberg,phillips-cts-trace}), however the simple nuclear case remained open both in the separable and in the non-separable settings.

The separable case remains a difficult open problem. AF algebras are necessarily isomorphic to their opposites, due to Elliott's classification theorem, and our results show that this cannot be recast as a result purely of a local approximation property. There has been major progress in the Elliott classification program recently, but the state-of-the art classification theorems all assume the UCT. Notably, we do not know if there are Kirchberg algebras which are not isomorphic to their opposites. If such an algebra exists, then it would necessarily be a counterexample to the UCT. More generally, both the Elliott invariant and the Cuntz semigroup of any \cstar-algebra $A$ are isomorphic to that of $A^{\textrm{op}}$.

%Non-separable AF algebra were studied by the first-named author and Katsura in \cite{farah-katsura-I,farah-katsura-II}, where it was shown among other things that the Elliott classification theorem for AF algebras does not extend to the non-separable setting,
%and the results in this paper add to the complexity of the picture.

The additional axiom we add to ZFC is Jensen's $\diamondsuit_{\aleph_1}$, discussed below in Section~\ref{section:diamond}, and our construction is motivated by the work of Akemann and Weaver from \cite{AkeWe:Consistency}, where they use  $\diamondsuit_{\aleph_1}$ to construct a counterexample to the Naimark problem. Our main theorem is:

\begin{theorem} \label{T.O2} Assume $\diamondsuit_{\aleph_1}$ holds. 
Then there exists a nuclear, simple, unital \cstar-algebra $A$ not isomorphic to its opposite
algebra. 
\end{theorem} 

In fact, we obtain the following strengthening. 

\begin{theorem} \label{T.O2+} Assume $\diamondsuit_{\aleph_1}$ holds
and $1\leq n\leq \aleph_0$ is given. 
Then there exists a \cstar-algebra $A$ with the following properties. 
\begin{enumerate}
\item $A$ is nuclear, simple, unital and of density character $\aleph_1$. 
\item $A$ is not isomorphic to its opposite
algebra. 
\item \label{I.O2+.3} $A$ has exactly $n$ unitarily nonequivalent irreducible representations. 
\item All automorphisms of $A$ are inner. 
\pushcounter
\end{enumerate}
In addition, one  can ensure that one of the following holds. 
\begin{enumerate}
\popcounter\item $A$ is an inductive limit of subalgebras isomorphic to the Cuntz algebra 
$\cO_2$. 
\item $A$ is an inductive limit of subalgebras isomorphic to 
full matrix algebras of the form $M_{2^n}(\bbC)$. 
\end{enumerate}
\end{theorem} 

By Glimm's theorem (see the remark in the second paragraph from the end of page 586 of \cite{Glimm-type-I}), 
every separable and simple \cstar-algebra with nonequivalent irreducible representations
has $2^{\aleph_0}$ nonequivalent irreducible representations.  
Item \eqref{I.O2+.3} above shows that the failure of   this dichotomy  
for nonseparable \cstar-algebras is relatively consistent with ZFC. 

The observation that the proof of \cite{AkeWe:Consistency} gives a nuclear counterexample to Naimark's problem is due to N. C. Phillips.  
We don't know whether a simple, nuclear \cstar-algebra not isomorphic to its opposite can be constructed in ZFC, and whether a counterexample to Naimark's problem can be constructed in ZFC.  Another problem raised by our proof of Theorem~\ref{T.O2+}  is whether a counterexample to Naimark's 
problem can have an outer automorphism. 

We use the following notation throughout. We count 0 as a natural number. If $\cY=\langle a_j: j\in \bbN\rangle$ is a sequence of elements in some set, we denote by $b^\frown \cY$ the sequence whose first element is $b$, and whose $j+1$ element is $a_j$. 

\section{Extending states}

This section contains technical lemmas which will be used in the induction step of our construction. We first give a modification of a lemma of Kishimoto, Lemma~\ref{L.Kishimoto}, and a toy version, Lemma~\ref{L.0}.

\begin{lemma}
\label{L.0}
Let $A$ be a a non-type I, 
separable, simple, unital  \cstar-algebra. Let $C$ and $D$ be non-zero hereditary subalgebras of $A$, and let $\e>0$. Let $n \geq 1$ and let $u_0,u_1,\ldots u_n$ be some elements in $A^+$. Then there exist positive elements $c \in C$ and $d \in D$ of norm $1$ such that $\|cu_kd\|<\e$ for $k=0,1,\ldots,n$.
\end{lemma}
\begin{proof}
We denote $A_{\infty} := l^{\infty}(\N,A)/C_0(\N,A)$, and we identify $A$ with the subalgebra given by constant sequences.  
As $A$ is not a continuous trace algebra, by \cite[Theorem 2.4]{akemann-pedersen}, the central sequence algebra $A_{\infty} \cap A'$ is nontrivial. 
Let $x \in A_{\infty} \cap A'$ be a self-adjoint element whose spectrum has more than one point. 
Since $A$ is simple, the \cstar-algebra generated by $x$ and $A$ inside of $A_{\infty}$ is isomorphic to $C(\sigma(x))\otimes A$, and therefore, if $y \in C^*(x)$ and $a \in A$ then $\|ya\| = \|y\|\|a\|$. 
Since $\sigma(x)$ has more than one point, we may pick $y,z \in C^*(x)_+$ with norm $1$ such that $yz = 0$. Pick  $(y_n)_{n \in \N} , (z_n)_{n \in \N} \in l^{\infty}(\N,A)_+$ which lift $y$ and $z$, respectively. Fix elements $c_0 \in C_+$ and $d_0 \in D_+$ of norm $1$. Then $\lim_{n \to \infty} \|c_0^{1/2}y_n c_0^{1/2}\| = \lim_{n \to \infty} \|d_0^{1/2}z_n d_0^{1/2}\| = 1$, and $ \lim_{n \to \infty}
\|c_0^{1/2}y_n c_0^{1/2} \cdot u_k \cdot d_0^{1/2}z_n d_0^{1/2}\| = \lim_{n \to \infty}
\|c_0y_n z_n u_k d_0\| = 0$. For all sufficiently large $n$, the elements $c = \frac{1}{\|c_0^{1/2}y_n c_0^{1/2}\|}  \cdot c_0^{1/2}y_n c_0^{1/2}$ and $d = \frac{1}{\|d_0^{1/2}z_n d_0^{1/2}\|}  \cdot d_0^{1/2}z_n d_0^{1/2}$ satisfy the requirements.
\end{proof}

\begin{lemma} \label{L.Kishimoto} 
Suppose $A$ is a non-type I, 
separable, simple, unital \cstar-algebra and suppose $\alpha$ 
is an antiautomorphism of $A$ or an outer automorphism 
of $A$. Then for any nonzero hereditary \cstar-subalgebra
$B$ of $A$ and every unitary $u\in A$ we have 
\[
\inf\{\|b u \alpha(b)\|: b\in B_+, \|b\|=1\}=0. 
\]
\end{lemma} 

\begin{proof} Since an automorphism of a simple \cstar-algebra
is outer if and only if its Connes spectrum is distinct from $\{1\}$, 
the case in which $\alpha$ is an outer automorphism is a special case of  
\cite[Lemma~1.1]{Kishimoto1981a}. 

Suppose $\alpha$ is an antiautomorphism and let $\alpha':=\Ad u\circ \alpha$.   
By \cite[Theorem~1]{hayashi2004kishimoto} 
we have $\inf\{\|b \alpha'(b)\|: b\in B_+, \|b\|=1\}=0$. 
But $\|b\alpha'(b)\|=\|b u \alpha(b)u^*\|=\|b u \alpha(b)\|$
and the conclusion follows. 
\end{proof}  

\begin{lemma} \label{L.crossed} 
Suppose $A$ is a  
separable, simple, unital \cstar-algebra. 
Suppose  $\cX$ and~$\cY$ are disjoint countable sets of unitarily nonequivalent pure states of $A$  
and suppose $E$ is an equivalence relation on $\cY$. Then there exists a separable simple unital \cstar-algebra~$C$ 
with the following properties. 
\begin{enumerate}
\item\label{I.crossed.1}  $A$ is a unital subalgebra of $C$. 
\item \label{I.crossed.2} 
Every $\psi\in \cY$ has a unique extension $\tilde \psi$ to a pure state of $C$. 
\item \label{I.crossed.3} If $\psi_0$ and $\psi_1$ are  in $\cY$ then 
$\psi_0\, E\, \psi_1$ if and only if 
$\tilde\psi_0$ and $\tilde\psi_1$ are unitarily equivalent pure states of $C$. 
\item \label{I.crossed.4} Every $\psi\in \cX$ has more than one extension to a pure state of $C$. 
\pushcounter
\end{enumerate}
In addition,  if $A\cong \cO_2$ then one can arrange $C\cong \cO_2$. 
\end{lemma} 

\begin{proof} 
We shall construct an automorphism $\beta$ of $A$ of infinite order such that 
the crossed product $C:=A\rtimes_\beta \bbZ$ is as required.  
By \cite[Theorem~2]{AkeWe:Consistency}, a pure state $\varphi$ of $A$ has a
unique extension to a pure state of $C$ if and only if $\varphi$ is nonequivalent 
to $\varphi\circ \beta^n$ for all $n\neq 0$. Since $A$ is non-type I and separable, 
by Glimm's theorem it has $2^{\aleph_0}$ nonequivalent pure states. 
We can therefore extend $\cY$ to ensure that every $E$-equivalence class is infinite
and that there are infinitely many equivalence classes. We can similarly assume $\cX$ is infinite.  
 Let $\pi_j^k$, for $j\in \bbZ$, be an enumeration of GNS representations 
 corresponding to states in the $k$-th $E$-equivalence class. 
  Let $\sigma_j$,  for $j\in \bbN$, be an enumeration of 
   GNS  representations corresponding to states in  $\cX$. All of these representations correspond
   to pure states and are therefore irreducible.  
 By  the extension of \cite{KiOzSa}  proved in  \cite[p. 7523--7524]{AkeWe:Consistency} there exists an automorphism 
 $\beta$ of $A$ such that 
 \begin{enumerate}
\popcounter
 \item \label{I.crossed.5} 
$\pi_j^k$ is equivalent to $\pi_l^m\circ \beta$
 if and only if $k=m$ and $j=l+1$. 
\item  \label{I.crossed.6}$\sigma_j$ is equivalent to $\sigma_j\circ \beta$ for all $j$. 
\end{enumerate}
By \cite[Theorem~3.1]{Kishimoto1981a}  
the crossed product $C:=A\rtimes_\beta\bbZ$ is simple. 
By \cite[Theorem~2]{AkeWe:Consistency} it 
satisfies \eqref{I.crossed.1}, 
\eqref{I.crossed.2}, and
\eqref{I.crossed.4}.  

To prove \eqref{I.crossed.3},  
fix $\psi_0$ and $\psi_1$ in $\cY$. If $\psi_0 \, E\, \psi_1$ then  
  \eqref{I.crossed.6} implies that the unique pure state extensions of $\psi_0$ and $\psi_1$ to $C$  
are equivalent. Now suppose $\psi_0$ and $\psi_1$ are  not $E$-related. Then $\psi_0$ and $\psi_1\circ \beta^n$ are inequivalent for all $n\in \bbZ$. 
To get a contradiction, suppose that the unique pure state extensions of $\psi_0$ and $\psi_1$ to $C$ are equivalent and let $v$ be a unitary in $C$  such that $\psi_0=\psi_1\circ \Ad v$. 
Let $u$ be the canonical unitary implementing $\beta$. 
Approximate $v$ up to $1/2$ by a finite linear combination $\sum_{n=-k}^k  c_n u^n$, where $c_n\in A$. 
Choose decreasing sequences $a_{j}, b_j$, for $j\in \bbN$, 
of positive elements of norm 1 
such that the $a_j$ excise $\psi_0$ and the 
$b_j$ excise $\psi_1$ (\cite[Proposition~2.2]{AkeAndPed}). Note that $\beta^n(b_j)$ excises $\psi_1\circ \beta^{-n}$ for all $n$.  
By \cite[Lemma~1]{AkeWe:Consistency}, for all $x\in A$  we have $\|a_j x \beta^n(b_j)\|\to 0$ as $j\to \infty$. 
Thus, for  $j$ large enough,  we have  
%$\|a_j\|=\psi_j(a_j)=1$ for $j=0,1$ and 
$\|a_j c_n \beta^n(b_j)\|<1/(4k+2)$ for all $-k\leq n\leq k$. Then $\|a_j  v b_j v^*\|=\|a_j  v b_j\|<1$. 
On the other hand, the Cauchy--Schwarz inequality  implies $\psi_0(a_j  v b_j v^*)=\psi_0(v b_j v^*)=\psi_1(b_j)=1$; contradiction. 

Finally, if $A\cong \cO_2$, then $C=A\rtimes_\beta \bbZ\cong \cO_2$. One way to see this is to note that by (\ref{I.crossed.5}) above,
  no non-zero power of $\beta$ is inner, therefore by \cite[Theorem 1]{nakamura} the automorphism $\beta$ has the Rokhlin property, hence by \cite[Theorem 4.4]{hirshberg-winter} we have $C \cong C \otimes \cO_2$, so by \cite[Theorem 3.8]{kirchberg-phillips} 
we have $C \cong \cO_2$. 
\end{proof} 

The following is a strengthening 
of \cite[Theorem 2.1]{Kishimoto1981a}.
\begin{lemma}
\label{lemma:uncountable-inequivalent-states}
Suppose $A$ is a non-type I, 
separable, simple, unital \cstar-algebra, and suppose   $\alpha$ is  an antiautomorphism,  or an outer automorphism. Then there exists a family $\cW$ of $2^{\aleph_0}$ 
pure states of $A$ such that $\varphi$ is not unitarily equivalent to $\varphi\circ\alpha$
for every $\varphi\in \cW$. 
\end{lemma}
\begin{proof}
The proofs in the case when  $\alpha$ is an outer automorphism 
and when $\alpha$ is an antiautomorphism differ very little and will be presented simultaneously.

Let $u_n$, for $n\in \bbN$, be an enumeration of a dense set of unitaries of $A$. 
By $\SEQ$ we denote the set of all finite sequences of $\{0,1\}$
 ordered by the  end-extension, denoted $s\sqsubset t$. 
 The empty sequence
 $\NULL$ is the minimal element of $\SEQ$, its immediate successors are $0$ and 
$1$, and the immediate successors of $s\in \SEQ$ are $s^\frown 0$ and 
$s^\frown 1$. The length of $s\in \SEQ$ is denoted $|s|$.  
%(purists may want to take note
%that, with the natural identification of  $s\in \SEQ$ with a partial function from $n\in \bbN$ 
%into $\{0,1\}$, the length of $s$  is equal to the cardinality of the  (graph of) $s$). 
  
Given $\delta \in (0,1/2)$, we claim that there exist  $a(s)$ and $e(s)$ in $A_+$, for $s \in \SEQ$ and $j=0,1$ 
 for $s\in \SEQ$, with the following properties:
 \begin{enumerate}
%\popcounter
\item\label{I.onestep.1} $\|a(s)\|=\|e(s)\|=1$. 
\item\label{I.onestep.2} $a(s) e(s^\frown j) = e(s^\frown j)$.
\item  \label{I.onestep.3}$e(s) a(s)=a(s)$.  
\item\label{I.onestep.4.1} $\|e(s^\frown 0) e(s^\frown 1)\|<\delta$. 
\item \label{I.onestep.5} $\|e(s^\frown 0) u_ke(s^\frown 1)\|<\delta$ for all $k\leq |s|$. 
\item \label{I.onestep.6} $\|a_{s} u_{|s|} \alpha(a_{s})\|<\delta$.
\end{enumerate}
The  family $\{e(s),a(s)\}_{s\in \SEQ}$ will be constructed by recursion. Define $f_t,g \colon [0,1] \to [0,1]$ for $t \in (0,1)$ as follows.
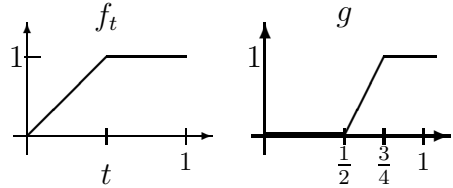
\begin{center}
\begin{picture}(275,65)
 \put(0,10){\vector(1,0){75}}
 \put(5,3){\vector(0,1){50}}
 \put(35,7){\line(0,1){6}}
 \put(65,7){\line(0,1){6}}
 \put(4,40){\line(1,0){6}}
\thicklines
 \put(5,10){\line(1,1){30}}
 \put(35,40){\line(1,0){30}}
 \put(1,40){\makebox(0,0){$1$}}
 \put(35,-8){\makebox(0,0)[b]{$t$}}
 \put(65,-3){\makebox(0,0)[b]{\footnotesize $1$\normalsize}}
 \put(35,55){\makebox(0,0){$f_t$}}
   \put(90,10){\vector(1,0){75}}
    \put(95,3){\vector(0,1){50}}
    \put(125,7){\line(0,1){6}}
    \put(140,7){\line(0,1){6}}
    \put(155,7){\line(0,1){6}}
    \put(154,40){\line(1,0){6}}
   \thicklines
    \put(95,11){\line(1,0){30}}
    \put(125,10){\line(1,2){15}}
    \put(140,40){\line(1,0){15}}
    \put(91,40){\makebox(0,0){$1$}}
    \put(125,-8){\makebox(0,0)[b]{$\frac{1}{2}$}}
     \put(140,-8){\makebox(0,0)[b]{$\frac{3}{4}$}}
    \put(155,-3){\makebox(0,0)[b]{\footnotesize $1$\normalsize}}
    \put(125,55){\makebox(0,0){$g$}}
\end{picture}
\end{center}
Notice that $f_{1/2} \cdot g = g$, and $\|f_t - \id\| = 1-t$. 
Fix $\e \in (0,1/2)$ such that whenever $x,y$ are positive contractions in some $C^*$-algebra and $z$ is any contraction such that $\|xzy\|<\e$ then $\|f_{1/2}(x)zf_{1/2}(y)\|<\delta$ and $\|g(x)zg(y)\|<\delta$. (This is done using polynomial approximations for $f_{1/2}$ and for $g$.)

Let $a(\NULL)=1$. 
Suppose $a(s)$ was chosen. 
By Lemma~\ref{L.0} applied to $n=|s|+1$ and the unitaries $u_k$ for $k\leq n$, 
there exist $h_0,h_1\in B(s)_+$  
such that $\|h_0\|=\|h_1\|=1$ and 
$\|h_0 u_{k} h_1\| < \e$.
 for all $k\leq |s|$. 
Let 
\[
e(s^\frown j):=f_{1/2}(h_j) .
\]
By Lemma~\ref{L.Kishimoto} there exists $a (s^\frown j)\in \overline{g(h_j)A g(h_j)}_+$ that satisfies
 $\|a(s^\frown j)\|=1$  
and $\|a(j) u_{|s|} \alpha(a(j))\|<\delta$. We may assume without loss of generality, that there exists a nonzero positive element $b(s^\frown j)$ with $a(s^\frown j) b(s^\frown j) = b(s^\frown j)$ (by replacing $a(s^\frown j)$ by $f_t(a(s^\frown j))$ for~$t$ sufficiently close to $1$ if need be).

 The family $\{e(s), a(s)\}_{s\in \SEQ}$ satisfying \eqref{I.onestep.1}--\eqref{I.onestep.6}
 can now be constructed by using a standard bookkeeping device. 
 Fix an enumeration $s_j$, for $j\in \bbN$, for $\SEQ$ such that $s_j \sqsubset s_k$
 implies $j<k$ (e.g. let $\{s\in \SEQ:  |s|=n\}$ be enumerated as $s_j$, for $2^{n-1}\leq j <2^n$). 
By using the above, one can recursively find $e(s_j)$ and $a(s_j)$ for $j\in \bbN$
 in the hereditary subalgebra on which all the
 elements of the form $e(s)$ and $a(s)$, where $s\sqsubset s_j$,  act  as the identity.

Denote the set of all infinite sequences of $\{0,1\}$ by $\xSEQ$. 
For $h\in \xSEQ$ let  $h\rs n$  denote the initial segment of $h$ of length $n$, for $n\in \bbN$. 
For  $h\in \xSEQ$ we have  $h\rs n\in \SEQ$ 
and
\[
\cF(h):=\{a(h\rs n): n\in \bbN\}
\]
is a sequence of elements of $A_+$ of norm 1 such that 
$$
a(h \rs n) a(h\rs (n+1))=a(h\rs (n+1))
$$ 
for all $n$. Hence 
\[
\{\zeta\in \cS(A)  : \zeta(a)=1\text{ for all }a\in \cF(h)\}
\]
is a face of $\cS(A)$. Let  $\zeta_h$ be an extreme point of this face; 
then $\zeta_h$ is a pure state of~$A$ satisfying   $\zeta_h(a(h\rs n))=1$ for all $n$. 
By \eqref{I.onestep.3} we have $\zeta_h(e(h\rs n))=1$ for all $n$ and thus,
by the 
Cauchy--Schwarz inequality, we have $\zeta_h(e(h\rs n)b )=\zeta_h(b)$ for all $b$ and for all $n$. 

We claim that the 
 states $\zeta_h$ and $\zeta_{h'}$ are not unitarily equivalent if $h\neq h'$. 
 Suppose otherwise. Then for some $j\in \bbN$ we have $\|\zeta_h-\zeta_{h'}\circ \Ad u_j\|<1/2$. 
 Fix  $n\geq j$ large enough to have $h\rs n\neq h'\rs n$. 
 By  \eqref{I.onestep.5} we have 
 $\|e(h\rs n) \Ad u_j  (e(h'\rs n))\| < \delta < 1/2$, 
 but $|\zeta_h(e(h\rs n) \Ad u_j  (e(h'\rs n))| =|\zeta_h(\Ad u_j  (e(h'\rs n))|>1/2$, 
 a contradiction. 
 
  By the same argument and \eqref{I.onestep.6}, 
  $\zeta_h$ is not equivalent to  $\zeta_h\circ \alpha$ for every $h\in \xSEQ$. 
  We should note that whether $\alpha$ be an automorphism 
  or an antiautomorphism, it  preserves the order structure of $A$ and
  it is an affine homeomorpism of $\cS(A)$ onto itself. 
  Therefore $\zeta_h\circ \alpha$ is a pure state of $A$. 
\end{proof}

%\section{A UHF counterexample} 
The next few technical lemmas will be used to construct a UHF example.

\begin{definition} \label{Def.Separated} 
Suppose $A$ is a separable UHF algebra. A family of pure states 
$\langle \varphi_n : n\in \bbN \rangle$
of $A$
will be called \emph{separated product states} if there exist
$\langle k(n) : n \in \bbN \rangle$, a map $\Phi$, subalgebras $A_n$, and projections $\langle p_{n,j} : n \in \bbN \, , \, j < n \rangle$ and $\langle q_n : n \in \bbN\rangle$ 
with  the following properties. 
\begin{enumerate}
\item  $k(n)\geq 1$, for $n\in \bbN$. 
\item  $\Phi\colon A\to \bigotimes_n M_{k(n)}(\bbC)$ is an isomorphism. 
\item $A_n:=\bigotimes_{j<n} M_{k(j)}(\bbC)$. 
\item  $p_{n,j}$, for $0\leq j<n$,  are orthogonal rank 1 projections 
in  $M_{k(n)}(\bbC)$, 
for all $n$, 
\item  $q_m\in A_m$ is a rank-1 projection, and 
\item $\varphi_m$ is the product state of $A_n\otimes \bigotimes_{j=m+1}^\infty M_{k(j)} (\bbC)$
uniquely determined by the requirement  that for all $l\geq 1$ we have 
\[
\varphi_m(q_m\otimes p_{m+1,m}\otimes p_{m+2,m}\otimes\dots \otimes p_{m+l,m})=1. 
\]
\end{enumerate}
\end{definition} 

\begin{lemma} \label{L.UHF.1} Suppose 
 $A$ is a UHF algebra and $\pi_n$ for $n\in \bbN$, are irreducible representations 
 of $A$.  Then the following are equivalent. 
\begin{enumerate}
\item\label{I.L1.1}  $\langle \pi_n : n\in \bbN\rangle$  are pairwise nonequivalent irreducible representations of $A$, 
\item \label {I.L1.2}There are separated product  states  $\varphi_n$, for $n\in \bbN$, such that $\pi_n$ is the GNS representation corresponding to $\varphi_n$
for all $n$. 
\end{enumerate}
\end{lemma} 

\begin{proof} Suppose 
  $\varphi_j$, for $j\in \bbN$, are separated product states of a UHF algebra. 
  For all $j\neq l$ and  $n\in \bbN$ there exists a projection $p\in A_n'\cap A$
  such that $\varphi_j(p)=0$ and $\varphi_l(p)=1$, 
  and therefore  \cite[Theorem~3.4]{Glimm:On} implies that 
$\varphi_l$ is not unitarily equivalent to $\varphi_j$ for $j\neq l$. 

 Now suppose $\pi_j$, for $j\in \bbN$, are as in \eqref{I.L1.2}.
 Let  $\psi_j$ be  a  pure state such that $\pi_j$ 
 is the GNS representation corresponding to $\psi_j$ for $j\in \bbN$. 
Let $\varphi_j$, for $j\in \bbN$, be a sequence of separated pure states of $A$. 
By \eqref{I.L1.1} these pure states are nonequivalent and 
 by  the extension of \cite{KiOzSa}  proved in  \cite[p. 7523--7524]{AkeWe:Consistency} 
(or, since $A$ is UHF,  by   \cite[Theorem~7.5]{FuKaKi})  there exists an automorphism $\beta$ of $A$ such that 
 $\varphi_j =\psi_j\circ \beta$ for all $j\in \bbN$, as required. 
\end{proof} 

% Theorem~\ref{L.UHF.1} could be also completed 
% by using an extension of  \cite[Theorem~3.7]{powers1967representations}
% to sequences of nonequivalent irreducible representations. 

We need the following variant of Lemma~\ref{L.crossed} 
for the CAR algebra, $M_{2^\infty}$. 

\begin{lemma} \label{L.UHF.crossed} Suppose $A\cong M_{2^\infty}$. 
Suppose  $\cX$ and~$\cY$ are disjoint countable sets of unitarily nonequivalent pure states of $A$  
and $E$ is an equivalence relation on $\cY$. Then there exists a separable simple unital \cstar-algebra $C$ 
with the following properties. 
\begin{enumerate}
\item $C\cong M_{2^\infty}$. 
\item\label{I.UHF.crossed.1}  $A$ is a unital subalgebra of $C$. 
\item \label{I.UHF.crossed.2} 
Every $\psi\in \cY$ has a unique extension $\tilde \psi$ to a pure state of $C$, 
\item \label{I.UHF.crossed.3} If $\psi_0$ and $\psi_1$ are  in $\cY$ then 
$\psi_0\, E\, \psi_1$ if and only if 
$\tilde\psi_0$ and $\tilde\psi_1$ are unitarily equivalent pure states of $C$. 
\item Every $\psi\in \cX$ has more than one extension to a pure state of $C$. 
\pushcounter 
\end{enumerate}
\end{lemma}

 \begin{proof} We shall first provide a proof in case when $E$ is the identity relation on $\cY$. 
 By Lemma~\ref{L.UHF.1} we may identify  $A$ with $\bigotimes_n M_{k(n)}(\bbC)$
 witnessing that the  pure states in $\cX\cup \cY$ 
are separated. 
Since $A\cong M_{2^\infty}$, for every $n$ 
there exists $l(n)\in \bbN$ such that 
$k(n)=2^{l(n)}$. We may assume that $k(n)>2n$ for all $n$. 
In $M_{k(n)}(\bbC)$ 
we have $n$ orthogonal rank 1 projections $p_{n,j}$, for $j\leq n$, each corresponding to a
unique state in $\cX\cup \cY$. Let $\cP$ be a maximal family of orthogonal rank 1 projections in $M_{k(n)}$ including $\{p_{n,j}: j\leq n\}$. 
Since $k(n)>2n$, we can find a permutation $\sigma$ of $\cP$ such that
\begin{enumerate}
\popcounter
\item $\sigma(p_{n,j})=p_{n,j}$ if and only if  $p_{n,j}$ corresponds to a pure state in $\cX$,
\item\label{I.UHF.crossed.5}  $\sigma(p_{n,j})\neq p_{n,k}$ if $p_{n,j}$ and $p_{n,k}$ correspond to distinct pure states in $\cY$, and 
\item $\sigma^2=\id_{\cP}$.    
\end{enumerate}
Let $u_n\in M_{k(n)}(\bbC)$ be an order 2 unitary such that $\Ad u_n (q)=\sigma(q)$ for all $q\in \cP$ and such that $\Tr(u_n) = 0$. (One can construct such a unitary by first considering a permutation matrix corresponding to $\sigma$, and noting that the number of $1$'s on the diagonal must be even; we then define $u_n$ to be a matrix obtained by starting out with this permutation matrix and replacing half of the $1$'s on the diagonal by $-1$'s.)
Note that the automorphism 
  $\beta:=\bigotimes_n \Ad u_n$ also satisfies $\beta^2=\id_A$. 
  
  Set $A_n$ as in Definition \ref{Def.Separated}. Each $A_n$ is $\beta$-invariant and we have $A\rtimes_\beta \bbZ/2\bbZ = \overline{\bigcup_n A_n\rtimes_{\beta|_{A_n}} \bbZ/2\bbZ}$. Note that $A_n\rtimes_{\beta|_{A_n}} \bbZ/2\bbZ \cong A_n \oplus A_n$, and the inclusion 
  \[
  A_n\rtimes_{\beta|_{A_n}} \bbZ/2\bbZ \to A_{n+1}\rtimes_{\beta|_{A_{n+1}}} \bbZ/2\bbZ \cong (A_n\rtimes_{\beta|_{A_n}} \bbZ/2\bbZ) \otimes M_{k(n)}
  \]
   is given by a direct sum of $k(n)/2$ copies of the identity map, and $k(n)/2$ copies of the map $a \oplus b \mapsto b \oplus a$. Thus, by considering the Bratteli diagram of this AF system, we see that $A \rtimes_{\beta} \bbZ/2\bbZ \cong M_{2^{\infty}}$.
  
By   \cite[Theorem~2]{AkeWe:Consistency}
a pure state $\varphi$ of $A$ has a unique extension to a pure state of $C$
 if and only if $\varphi$ and $\varphi\circ \beta$ are not unitarily equivalent. 
 By the choice of $u_n$ and $\beta$, a pure state $\varphi\in \cX\cup \cY$ has a unique 
 extension to a pure state of $C$ if and only if $\varphi\in \cX$. 
  If $\varphi$ and $\psi$ are distinct and  belong to $\cY$, then by  \eqref{I.UHF.crossed.5} 
for every finite-dimensional subalgebra $B$ of 
$C$  there exists a projection $p\in B'\cap C$ (one can choose it of the form $q + \sigma(q)$ for $q$ which corresponds to $\psi$)
  such that $\tilde\varphi(p)=0$ and $\tilde\psi(p)=1$. Therefore  \cite[Theorem~3.4]{Glimm:On} implies that  
$\tilde\varphi$   is not unitarily equivalent to $\tilde \psi$.

We now consider the case when $E$ is a nontrivial equivalence relation on $\cY$. 
Enumerate the $i$-th  $E$-equivalence class as $\langle \zeta_j^i: j<n\rangle$, for some  $1\leq n\leq \aleph_0$. 
In the above construction 
  there is sufficient room for us to 
  choose the symmetry $\sigma$ so the 
  resulting automorphism $\beta$ satisfies $\zeta_0^i \circ \beta=\zeta_1^i$ for all $i$.  
The resulting crossed product,~$A_1$, is isomorphic to $M_{2^\infty}$,
every $\zeta_j^i\in \cY$ has a unique extension $\tilde\zeta_j^i$ to a pure state of $A_1$, 
and $\tilde\zeta_j^i$ is equivalent to $\tilde\zeta_k^l$ if and only if $i=l$ and $\max(j,k)\leq 1$.  
We can now apply  this construction to $A_1$,  
with $\cX:=\emptyset$, 
$\cY:=\{\tilde \zeta_j^i: j\geq 1\}$ and $E$ defined by  
$\tilde \zeta_j^i E \tilde \zeta_k^l$ if and only if $i=l$ and $\min(j,k)\geq 1$
and obtain crossed product $A_2$. 
After at most $\aleph_0$ steps all $E$-equivalence classes will be taken 
care of. The inductive limit $C$ of $A_n$ is, by the classification of AF algebras, 
isomorphic to $M_{2^\infty}$  and it has all the required properties. 
\end{proof} 

The following lemma serves as the inductive step in our construction.

\begin{lemma} \label{L.onestep} 
Suppose $A$ is a non-type I, 
separable, simple, unital \cstar-algebra 
and let  $\cY$ be a countable set of pure states of $A$. Let $\zeta$ be a pure state of $A$ which is not unitarily equivalent to any of the states in $\cY$.
Suppose $\alpha$ is  an antiautomorphism,  or an outer automorphism, 
of $A$.
Then there exist a separable simple unital \cstar-algebra $C$ 
and a pure state $\psi$ of $C$ 
such that:
\begin{enumerate}
\item $A$ is a unital \cstar-subalgebra of $C$. 
\item Each $\varphi\in \cY$ has a unique extension to a pure state of $C$, and those unique extensions are pairwise unitarily inequivalent.
\item $\zeta$ has a unique extension to a pure state in $C$ which is unitarily equivalent to the extension of some pure state from $\cY$.
\item $\psi$ is the unique extension of some pure state in $\cY$. 
\item $\alpha$ cannot be extended to an antiautomorphism or an 
automorphism of $C$.
\item \label{I.onestep.4} 
If a \cstar-algebra $D$
has $C$ as a  subalgebra and $\psi$ has a unique state extension to~$D$ then   
  $\alpha$ cannot be extended to an antiautomorphism or an 
automorphism of~$D$. 
\pushcounter
\end{enumerate}
In addition, if $A\cong \cO_2$ then we can arrange $C\cong \cO_2$, and if $A \cong M_{2^{\infty}}$ then we can arrange $C \cong M_{2^{\infty}}$.
\end{lemma}

\begin{proof} 
Again, the proofs in the case in which  $\alpha$ is an outer automorphism 
and when $\alpha$ is an antiautomorphism differ very little and will be presented simultaneously.  
We note in passing that our assumptions imply that $A$ is nonabelian, hence 
 an automorphism of $A$ cannot 
be extended to an antiautomorphism of $C$ and vice versa; however this is unimportant for the proof. 
 
  Since the given set $\cY$ of pure states is countable, by Lemma \ref{lemma:uncountable-inequivalent-states}, we can choose a pure state $\psi_0$ such that
   for any $\varphi\in \cY \cup \{\zeta\}$,
 neither $\psi_0$ nor $\psi_1:=\psi_0\circ \alpha$
  is unitarily equivalent to $\varphi$.
  Let $\cY':=\cY\cup \{\zeta,\psi_0\}$, and define an equivalence relation 
  $E$ on $\cY'$ such that $\zeta\, E\, \varphi$ and $\psi_0\, E\, \varphi$ for some $\varphi\in \cY$, and all other elements of $\cY'$ are equivalent via~$E$ only to themselves. We then apply Lemma~\ref{L.crossed} or Lemma~\ref{L.UHF.crossed}  to $\cX=\{\psi_1\}$ 
   and $\cY'$  
  to obtain a  \cstar-algebra $C$ (with $C\cong A$ if $A$ is $M_{2^\infty}$ or $\cO_2$)  
  such that $\psi_0$, $\zeta$ and all  $\varphi\in \cY$ have unique pure state extensions to~$C$, 
    $\psi_1$ has 
  multiple state extensions to~$C$, and the unique extensions of $\psi_0$ and $\zeta$ are
  equivalent to the unique extension of some $\varphi\in \cY$; 
  the latter state is $\psi$ as in \eqref{I.onestep.4}.
  
%  If we started out with the CAR algebra, then we obtain 
%  the analogous situation by applying Lemma ~\ref{L.UHF.crossed} twice, once ensuring that $\psi_0$ and $\psi_1$ extend as above while $\zeta$ extends uniquely to a state which is inequivalent to the extensions of all the states mentioned, and again, this time making the unique extension of the (extension of) $\zeta$ equivalent to the unique extension of some element in $\cY$.
  
Suppose $D$ is a \cstar-algebra that has $C$ as a \cstar-subalgebra, 
and assume that $\alpha$ extends to $\tilde\alpha$ which is an 
 automorphism or an antiautomorphism of $D$. 
 If $\psi$ has a unique state extension $\tilde \psi$ to $D$, then $\tilde\psi\circ\tilde \alpha$
 is the unique extension of $\psi_1$ to $D$. As $\psi_1$ has multiple state extensions to $C$ 
 this is a contradiction, and therefore \eqref{I.onestep.4} holds. 
 \end{proof} 

\section{Diamond and the construction} 
\label{section:diamond}

A subset $\cC$ of $\aleph_1$ is called \emph{closed and unbounded} (\emph{club})
if for every $\eta<\aleph_1$ there exists~$\xi \in \cC$ such that $\xi>\eta$, and 
for every countable $X\subseteq \cC$ we have $\sup(X)\in \cC$ (see \cite[\S III.6]{kunen2011set}).  
A subset $\cS$ of~$\aleph_1$ is \emph{stationary} if it intersects every club nontrivially. 
Since the intersection of two clubs (and even countably many clubs) 
is a club, the intersection of a stationary set with a club is again stationary.
We shall use von Neumann's~ definition of an ordinal as the set of all smaller ordinals.  
 
Jensen's $\doo$ asserts that  
there exists a family of sets $S_\xi$, for $\xi<\aleph_1$, such that 
\begin{enumerate}
\item $S_\xi\subseteq \xi$ for all $\xi<\aleph_1$,  and 
\item for every $X\subseteq \aleph_1$ the set $\{\xi: X\cap \xi=S_\xi\}$
is stationary. 
\end{enumerate}
This combinatorial principle is true in G\"odel's constructible universe $L$  
  (see e.g. \cite[\S III.7.13]{kunen2011set}) and 
  is therefore  
relatively consistent with ZFC. A much easier fact is that it 
implies the Continuum  Hypothesis (see e.g. \cite[III.7.2]{kunen2011set}). 
 
Although  $\doo$ captures subsets of $\aleph_1$,   
it is well-known among  logicians that $\doo$ implies its self-strengthening
which  captures countable (or separable) subsets of any algebraic structure in 
countable signature of cardinality $\aleph_1$. This extends to metric structures. Since we could not
 find a reference for this fact in the 
literature, we work out the details in case of \cstar-algebras equipped with some additional structure.

Suppose $A$ is a \cstar-algebra with a given sequence of states $\cY=\langle \varphi_j: j\in \bbN\rangle$ and 
a linear isometry $\alpha\colon A\to A$. (We are interested in the case when $\alpha$ 
is an automorphism or an antiautomorphism.)
Suppose we are given 
a dense subset of $A$, $\cA:=\{a_\xi: \xi<\theta\}$, indexed by an ordinal $\theta$.
In addition suppose that $\cA$ is closed under $+,\cdot, {}^*$,  $\alpha$, 
and multiplication by the complex rationals, $\bbQ+i\bbQ$.  
Consider the following subsets of~$\theta^k$, for $1\leq k\leq 3$ and of 
$\theta\times \bbQ$: 
\begin{enumerate}
\item $\cA(+):=\{(\xi,\eta, \mu)\in \theta^3: a_\xi+a_\eta=a_\mu\}$, 
\item $\cA(\cdot):=\{(\xi,\eta, \mu)\in \theta^3: a_\xi a_\eta=a_\mu\}$,  
\item $\cA(^*):=\{(\xi,\eta)\in \theta^2: a_\xi^* =a_\eta\}$, 
\item $\cA(\|\cdot\|):=\{(\xi, r)\in \theta\times \bbQ_+: \|a_\xi\|\geq r\}$, 
\item $\cA(\bbC):=\{(\xi,\eta) \in \theta^2: a_{\xi} = ia_{\eta}\}$, 
\item $\cA(\varphi_j):=\{(\xi, r)\in \theta\times \bbQ: 
\varphi_j(a_\xi^*a_\xi)\geq r\}$, for $j\in \bbN$, 
\item $\cA(\alpha):=\{(\xi, \eta)\in \theta^2: \alpha(a_\xi)=a_\eta\}$. 
\end{enumerate}
This countable family of  sets uniquely determines
a  countable  normed algebra over $\bbQ+i\bbQ$ 
 whose completion 
is isomorphic to $A$. It also uniquely determines 
both $\alpha$ and the sequence $\cY$. 
We say that the  structure 
$(A,\cA,\alpha,\varphi: \varphi\in \cY)$ 
is \emph{coded} by $\fX:=\langle \cA(\bullet): \bullet\in \{+, \cdot, {}^*, \|\cdot\| , \bbC, \alpha, \varphi: \varphi\in \cY\}\rangle$
and construe the latter as 
a subset of 
\[
\bbX(\theta):=\theta^3\sqcup \theta^3\sqcup\theta^2\sqcup \theta\times\bbQ 
\sqcup \theta^2
\sqcup \theta\times\bbQ\times \cY
\sqcup \theta^2. 
\]
Clearly  $\bbX(\theta)$ and $\theta$ have the same cardinality for any infinite $\theta$.

A nested transfinite sequence $A_\xi$, for $\xi<\aleph_1$, 
of \cstar-algebras is said to be \emph{continuous} if for every limit ordinal $\eta<\aleph_1$ we have 
$A_\eta=\overline{\bigcup_{\xi<\eta} A_\xi}$. 
%The following consequence of $\doo$ is well known, but we do not have a reference.

\begin{lemma} \label{L.diamond} 
$\doo$ implies that  there exists a family 
 $\{T_\xi\}_{\xi<\aleph_1}$ such that: 
\begin{enumerate}
\item $T_\xi\subseteq \bbX(\xi)$ for all $\xi<\aleph_1$,  
\item \label{I.diamond.2} for every continuous nested family $\{A_\xi\}_{\xi<\aleph_1}$ 
of separable \cstar-algebras, for any enumeration $\{a_\xi| \xi<\aleph_1\}$ of 
$A=\underset{\longrightarrow}{\lim} A_\xi$, for any countable set $\cY$ of pure states of $A$ and  for any
linear isometry $\alpha$ of $A$ onto $A$,  the set 
of all $\theta<\aleph_1$ such that 
\begin{enumerate}
\item $\varphi\rs A_\theta$ is pure for all $\varphi\in \cY$, 
\item $\alpha(A_\theta)=A_\theta$, and 
\item $T_\theta$ codes the structure 
$(A_\theta, \{a_\xi: \xi<\theta\},\alpha\rs{A_\theta},\varphi\rs{A_\theta}: \varphi\in \cY)$ 
\end{enumerate}
\end{enumerate}
is stationary. 
\end{lemma}

\begin{proof} 
Fix a   bijection 
$f\colon \aleph_1\to \bbX(\aleph_1)$. 
Writing $f[X]:=\{f(x): x\in X\}$, define $g\colon \aleph_1\to \aleph_1$ by 
$g(\xi):=\min\{\eta: f[\xi]\subseteq \bbX(\eta), f^{-1}[\bbX(\xi)]\subseteq \eta\}$. 
(Since every countable subset of $\aleph_1$ is bounded, $g$ is well-defined.)
The set of fixed points of $g$, $\cC:=\{\theta<\aleph_1: g[\theta]=\theta\}$, 
is a club (\cite[Lemma~III.6.13]{kunen2011set})
and $\cC\subseteq \{\theta<\aleph_1: f[\theta]=\bbX(\theta)\}$. 
%includes a club (\cite{kunen2011set}). \marginpar{Where in \cite{kunen2011set}?}
Let $\{S_{\xi}\}_{\xi < \aleph_1}$ be a family of sets as in the definition of $\doo$.
We claim that   $T_\xi:=f[S_\xi]$, for $\xi\in \cC$, and $T_\xi:=\emptyset$, for $\xi\notin \cC$, 
 are as required. 
(Many of the $T_\xi$ don't code anything resembling a \cstar-algebra, but this is of no concern for us.)

Suppose $A=\underset{\longrightarrow}{\lim} A_\xi$, $\cY$, $\alpha$,   and $\{a_\xi: \xi<\aleph_1\}$ 
are as in \eqref{I.diamond.2}. Set $\cA_\theta:=\{a_\xi: \xi<\theta\}$. 
Note that the set
\[
\cC_0:=\{\theta<\aleph_1: \cA_\theta\text{ is a dense $\bbQ+i\bbQ$ subalgebra of $A_\theta$}\}
\]
is a club. 
Since the intersection 
of countably many clubs is a club,   \cite[Lemma~4]{AkeWe:Consistency} implies that 
\[ \cC_1:=\{\theta\in \cC_0: \varphi_j\rs{A_\theta} \; \text{is pure 
for all $j\in \bbN$ and} \;  \alpha[A_\theta]=A_\theta\}
\]
is also a club. Let $\fX\subseteq \bbX(\aleph_1)$ be the code of 
 $(A, \cA,\alpha,\varphi: \varphi\in \cY)$ and with $f$ used to define $T_\xi$, 
 let $X:=f^{-1}(\fX)$. By $\doo$, the set $\{\theta : X \cap \theta = S_{\theta}\}$ is stationary, and therefore so is 
  its intersection with $\cC_1$. But $\{\theta : X \cap \theta = S_{\theta}\} \cap \cC_1$ is precisely the set of ordinals $\theta$ which satisfy (\ref{I.diamond.2}), as required.
 \end{proof}

\begin{proof}[Proof of Theorem~\ref{T.O2+}]  
We construct a continuous 
nested sequence $\{A_{\eta} : \eta < \aleph_1\}$ of simple, separable unital and nuclear \cstar-algebras 
and inequivalent 
pure states $\varphi_\eta^j$, for $j<n$, of $A_\eta$, such that $\varphi_\eta^j$ and $\varphi_\xi^j$ agree on $A_\xi$ 
if $\xi<\eta$. 
Since $\doo$ implies the Continuum Hypothesis, each $A_\eta$ as well as $\bigcup_{\eta<\aleph_1} A_\eta$
will be of cardinality $\aleph_1$. We shall choose an enumeration $A_\eta=\{b_\eta^\xi: \xi<\aleph_1\}$
for every  $\eta$ and a countable dense subset  
$\cA_\eta=\{a_\xi^\eta: \xi<\eta\}$ of $A_\eta$ for every limit ordinal $\eta$ such that 
  \begin{enumerate}
\item \label{I.1} $\cA_\eta$ is  closed under $+,\cdot, {}^*$,  and multiplication by the complex 
  rationals, $\bbQ+i\bbQ$, 
    \item \label{I.2} $a_\xi^\zeta=a_\xi^\eta$ if $\xi<\zeta<\eta$
    and   $\zeta$ and $\eta$  are   limit ordinals, 
  \item \label{I.3} $\{b_\zeta^\xi : \max\{\xi,\zeta\}<\eta\} \subseteq \cA_\eta$. 
  \end{enumerate} 
We begin with $A_0 = \cO_2$ or $A_0 = M_{2^{\infty}}$ and any fixed (finite or infinite) sequence $\langle \varphi_0^j : j < n\rangle$ of  inequivalent pure states of $A_0$.  

If $\theta$ is a limit ordinal then we let $A_\theta := \lim_{\xi<\theta} A_\xi$ and let $\varphi_\theta^j$ be the unique state extending 
all $\varphi_\xi^j$ for $\xi<\theta$ for $j<n$; this state is necessarily pure. If in addition $\theta$ is a limit of limit ordinals, then 
$\cA_\theta$ is already uniquely determined and conditions \eqref{I.2} and \eqref{I.3} for $\zeta<\eta<\theta$ 
imply the corresponding conditions for $\eta<\theta$. 
If $\theta$ is a limit ordinal, but not a limit of limit ordinals, then the supremum of limit ordinals $<\theta$ is the largest 
limit ordinal below $\theta$; we denote it by $\eta$. Then the set  $\{\xi: \eta\leq \xi<\theta\}$ is infinite.  
Since $A_\theta$ is separable and the set on the left-hand side of \eqref{I.3} is countable, 
 $\cA_\theta$ can be defined so that it satisfies the requirements. 

Now suppose $\theta$ is a successor ordinal, say $\theta=\xi+1$.  
To proceed from $A_\xi$ to $A_{\xi+1}$, we first check whether there exists an 
outer automorphism or an antiautomorphism 
 $\alpha$ of~$A_\xi$, pure state $\psi$ of $A_\xi$, and (if $n$ is finite) an extension 
 of $\langle \varphi_\xi^j:j<n\rangle$ to an infinite sequence~$\cZ$    such that 
$(A_\xi, \cA_\xi, \psi^\frown \cZ, \alpha)$ is coded by $T_\xi$. 
If so, let $A_{\xi+1}$ be the \cstar-algebra $C$ 
given by 
Lemma~\ref{L.onestep} in which the unique extension of $\psi$ is unitarily equivalent to a unique extension of some $\varphi_\xi^j$. 
Let $\varphi_{\xi+1}^j$ be the  unique extension of $\varphi_\xi^j$, for $j<n$. 
If $T_\xi$ does not code 
such $(A_\xi, \cA_\xi, \psi^\frown \cZ, \alpha)$,  let $A_{\xi+1}:=A_{\xi}$.  
This describes the construction. 

Let~$A$ be the inductive limit of this nested sequence. It is nuclear, simple and unital, being the inductive limit of 
simple nuclear \cstar-algebras with unital connecting maps.
 Using \eqref{I.2} we can  write $a_{\xi} := a_\xi^\zeta$ 
for $\zeta$ being any limit ordinal greater than $\xi$. 
 Since $A=\bigcup_\xi A_\xi$ by \eqref{I.3} we have  $A= \{a_\xi: \xi<\aleph_1\}$.

 The sequence of pure state extensions $\varphi_{\theta}^j$ defines 
  $n$ inequivalent pure states $\varphi^j$, for $j<n$, of $A$.
  These states have the property that $\varphi^j$ is a unique extension of $\varphi_\theta^j$ 
  to $A$, for every $\theta<\aleph_1$. 
     If $n$ is finite let $\cZ$ be any infinite sequence of pure 
     states of $A$ extending $\langle \varphi^j: j<n\rangle$. 

Suppose $A_0\cong \cO_2$ and  $A_\xi\cong \cO_2$ for all $\xi<\theta$. If $\theta=\xi+1$ then $A_\theta\cong \cO_2$ 
 since it was obtained by using Lemma~\ref{L.onestep}. 
 If $\theta$ is a limit ordinal  then \cite[Corollary~5.1.5]{Ror:Classification} implies $A_\theta\cong \cO_2$. 
 Therefore by induction $A_\xi\cong \cO_2$ for all $\xi<\aleph_1$. 
 Likewise, if $A_\xi\cong M_{2^{\infty}}$ for all $\xi<\theta$ then 
   $A_\theta\cong  M_{2^{\infty}}$ by the classification of AF algebras (noting that the inclusion maps all induce an isomorphism on the $K_0$ groups).
Since $A$ has density character $\aleph_1$, it is an inductive limit of full matrix algebras
by \cite[Theorem~1.3 (1)]{farah-katsura-I}.

  Suppose that $A$ has an 
 antiautomorphism or an outer automorphism 
 $\alpha$ and let $\varphi$ be any pure state of $A$. 
 Then there exists $\theta<\aleph_1$ such that  
 $(A_\theta,\cA_\theta,\varphi^\frown \cZ, \alpha\rs {A_\theta})$ was coded by $T_\theta$ 
  at stage $\theta$. 
 Hence $A_{\theta+1}$ was produced by using Lemma~\ref{L.onestep} 
 and there exists $j<n$  such that 
$\alpha\rs{A_{\theta}}$ cannot be  extended to an 
antiautomorphism or  an outer automorphism of any \cstar-algebra which contains $A_{\xi+1}$ and
to which $\varphi_{\xi+1}^j$ has a unique state extension. 
By  construction this state has a unique extension to $A_\eta$ for all $\eta\geq \xi+1$ and 
therefore it has a unique extension to $A$. 
But $\alpha$ clearly extends $\alpha\rs{A_\theta}$; contradiction. 

We already know that  $A$ has at least $n$ inequivalent pure states. 
Let  $\psi$ be any pure state of $A$. With $\alpha=\id_A$, 
 there exists $\theta<\aleph_1$ such that 
 $(A_\theta, \cA_\theta,\alpha\rs{A_\theta},\varphi\rs{A_\theta})$ was coded by $T_\theta$ 
  at stage $\theta$. 
 Hence $A_{\theta+1}$ was produced by using Lemma~\ref{L.onestep}
 %\marginpar{We have to reformulate Lemma~\ref{L.onestep} - fortunately the proof is ok}
and $\varphi\rs {A_{\theta}}$ has a unique extension to $A_{\theta+1}$ equivalent to 
$\varphi_{\theta+1}^j$ for some $j<n$. Since $\varphi^j$ is the unique extension
of the latter to a state of $A$, we conclude that $\psi$ is equivalent to $\varphi^j$. 
Since $\psi$ was arbitrary, we conclude that every pure state of $A$ is equivalent to some $\varphi^j$, for $j<n$, and therefore $A$ has exactly $n$ inequivalent pure states. 
 \end{proof}     
\begin{remark}
The AF algebra we constructed is not isomorphic to an (uncountable) infinite tensor power of copies of $M_2$ (or $M_n$). To see that, notice that an infinite tensor product of matrix algebras is the complexification of a real $C^*$-algebra (namely, the corresponding infinite tensor product of $M_2(\bbR)$). A complexification of a real $C^*$-algebra is always isomorphic to its opposite (any real $C^*$-algebra is isomorphic to its opposite via the $*$ map, which is $\bbR$-linear, which one can then complexify).
\end{remark}
\begin{remark}
Our construction is $C^*$-algebraic in nature. It does, however, raise the analogous question for von-Neumann algebras: is there a hyperfinite factor (with non-separable predual) which is not isomorphic to its opposite?
More concretely, our AF example has unique trace. Let $M$ be the weak closure of its image under the GNS representation. Is $M$ isomorphic to its opposite? A peculiar hyperfinite II$_1$ factor with no nontrivial
central sequences was constructed using the Continuum Hypothesis in \cite{FaHaKaTi}. 
  \end{remark}

\bibliographystyle{plain}
\bibliography{opposite}

\begin{thebibliography}{10}

\bibitem{AkeAndPed}
C.~A. Akemann, J.~Anderson, and G.~K. Pedersen.
\newblock Excising states of \cstar-algebras.
\newblock {\em Canad. J. Math.}, 38(5):1239--1260, 1986.

\bibitem{akemann-pedersen}
C.~A. Akemann and G.~K. Pedersen.
\newblock Central sequences and inner derivations of separable {$C^{\ast}
  $}-algebras.
\newblock {\em Amer. J. Math.}, 101(5):1047--1061, 1979.

\bibitem{AkeWe:Consistency}
C.~A. Akemann and N.~Weaver.
\newblock Consistency of a counterexample to {N}aimark's problem.
\newblock {\em Proc. Natl. Acad. Sci. USA}, 101(20):7522--7525, 2004.

\bibitem{connes}
A.~Connes.
\newblock A factor not anti-isomorphic to itself.
\newblock {\em Ann. Math. (2)}, 101:536--554, 1975.

\bibitem{FaHaKaTi}
I.~Farah, D.~Hathaway, T.~Katsura, and A.~Tikuisis.
\newblock A simple \cstar-algebra with finite nuclear dimension which is not
  {Z}-stable.
\newblock {\em M\"unster J. Math.}, 7(2):515--528, 2014.

\bibitem{farah-katsura-I}
I.~Farah and T.~Katsura.
\newblock Nonseparable {UHF} algebras {I}: {D}ixmier's problem.
\newblock {\em Adv. Math.}, 225(3):1399--1430, 2010.

\bibitem{FuKaKi}
H.~Futamura, N.~Kataoka, and A.~Kishimoto.
\newblock Homogeneity of the pure state space for separable
  {$C^\ast$}-algebras.
\newblock {\em Internat. J. Math.}, 12(7):813--845, 2001.

\bibitem{Glimm:On}
J.~G. Glimm.
\newblock On a certain class of operator algebras.
\newblock {\em Trans. Amer. Math. Soc.}, 95:318--340, 1960.

\bibitem{Glimm-type-I}
J.~G. Glimm.
\newblock Type {I} {$C^{\ast} $}-algebras.
\newblock {\em Ann. of Math. (2)}, 73:572--612, 1961.

\bibitem{hayashi2004kishimoto}
T.~Hayashi.
\newblock A {K}ishimoto type theorem for antiautomorphisms with some
  applications.
\newblock {\em Internat. J. Math.}, 15(5):487--499, 2004.

\bibitem{hirshberg-winter}
I.~Hirshberg and W.~Winter.
\newblock Rokhlin actions and self-absorbing {$C^*$}-algebras.
\newblock {\em Pacific J. Math.}, 233(1):125--143, 2007.

\bibitem{kirchberg-phillips}
E.~Kirchberg and N.~C. Phillips.
\newblock Embedding of exact {$C^*$}-algebras in the {C}untz algebra
  {$\mathcal{O}_2$}.
\newblock {\em J. Reine Angew. Math.}, 525:17--53, 2000.

\bibitem{Kishimoto1981a}
A.~Kishimoto.
\newblock {Outer Automorphisms and Reduced Crossed Products of Simple
  \cstar-Algebras}.
\newblock {\em Commun. Math. Phys.}, 81:429--435, 1981.

\bibitem{KiOzSa}
A.~Kishimoto, N.~Ozawa, and S.~Sakai.
\newblock Homogeneity of the pure state space of a separable {$C\sp
  *$}-algebra.
\newblock {\em Canad. Math. Bull.}, 46(3):365--372, 2003.

\bibitem{kunen2011set}
K.~Kunen.
\newblock {\em Set theory}, volume~34 of {\em Studies in Logic}.
\newblock College Publications, London, 2011.

\bibitem{nakamura}
H.~Nakamura.
\newblock Aperiodic automorphisms of nuclear purely infinite simple
  {$C^*$}-algebras.
\newblock {\em Ergodic Theory Dynam. Systems}, 20(6):1749--1765, 2000.

\bibitem{phillips-cts-trace}
N.~C. Phillips.
\newblock Continuous-trace {$C^*$}-algebras not isomorphic to their opposite
  algebras.
\newblock {\em Internat. J. Math.}, 12(3):263--275, 2001.

\bibitem{phillips-PAMS}
N.~C. Phillips.
\newblock A simple separable {$C^*$}-algebra not isomorphic to its opposite
  algebra.
\newblock {\em Proc. Amer. Math. Soc.}, 132(10):2997--3005 (electronic), 2004.

\bibitem{phillips-viola}
N.~C. Phillips and M.~G. Viola.
\newblock A simple separable exact {${\rm C}^*$}-algebra not anti-isomorphic to
  itself.
\newblock {\em Math. Ann.}, 355(2):783--799, 2013.

\bibitem{Ror:Classification}
M.~R{\o}rdam.
\newblock {\em Classification of nuclear {\cstar}-algebras}, volume 126 of {\em
  Encyclopaedia of Math. Sciences}.
\newblock Springer-Verlag, Berlin, 2002.

\bibitem{rosenberg}
J.~Rosenberg.
\newblock Continuous-trace algebras from the bundle theoretic point of view.
\newblock {\em J. Austral. Math. Soc. Ser. A}, 47(3):368--381, 1989.

\end{thebibliography}
\end{document}